\documentclass[10pt,3p]{elsarticle}

\usepackage{amssymb, amsthm, amsmath}
\usepackage{mathrsfs}
\usepackage{hyperref}
\usepackage{xcolor}
\usepackage{tikz-cd}
\usepackage{bm,bbm}
\usepackage{array}

\usepackage{etoolbox}

\theoremstyle{plain}
\newtheorem{theorem}{Theorem}[section]
\newtheorem{proposition}[theorem]{Proposition}
\newtheorem{lemma}[theorem]{Lemma}
\newtheorem{corollary}[theorem]{Corollary}

\theoremstyle{definition}
\newtheorem{definition}[theorem]{Definition}
\newtheorem{example}[theorem]{Example}
\newtheorem{remark}[theorem]{Remark}

\makeatletter
\newcommand\colim{\mathop{\mkern2mu{\operator@font colim}}}
\makeatother

\newcommand*\cocolon{%
        \nobreak
        \mskip6mu plus1mu
        \mathpunct{}%
        \nonscript
        \mkern-\thinmuskip
        {:}%
        \mskip2mu
        \relax
}

\newcommand{\boxa}{\scalebox{0.7}{\raisebox{1pt}{$\Box$}}}

\newcommand{\N}{\mathbb{N}} 
\newcommand{\Nt}{\mathbb{N}^{\infty}} 
\newcommand{\Z}{\mathbb{Z}} 
\newcommand{\2}{\bm{2}} 

\renewcommand{\S}{\mathbf{S}} 
\newcommand{\p}[1]{\widehat{#1}} 
\newcommand{\M}{\mathbf{M}} 
\newcommand{\C}{\mathbf{C}} 
\renewcommand{\d}[1]{{#1}^{\downarrow}} 
\newcommand{\K}[1]{\mathbf{T}^{#1}} 
\renewcommand{\t}{\tau} 
\newcommand{\w}[1]{\widehat{#1}} 
\newcommand{\down}{{\downarrow}}
\newcommand{\up}{{\uparrow}}

\renewcommand{\leq}{\leqslant}

\renewcommand{\phi}{\varphi}

\newcommand{\ev}{\mathrm{ev}} 

\newcommand{\epifin}{\twoheadrightarrow_f} 
\newcommand{\nat}{\Rightarrow} 

\newcommand{\G}{G} 
\newcommand{\V}{\mathcal{V}} 
\renewcommand{\P}{\mathcal{P}} 
\newcommand{\Pfin}{\mathcal{P}_f} 

\newcommand{\Set}{\mathbf{Set}} 
\newcommand{\Setfin}{\mathbf{Set}_f} 
\newcommand{\Bool}{\mathbf{BA}} 
\newcommand{\BStone}{\mathbf{BStone}} 
\newcommand{\Alg}[2]{{#1}^{#2}} 
\newcommand{\Algfin}[2]{{#1}_f^{#2}} 

\begin{document}

\makeatletter
\patchcmd{\ps@pprintTitle}
  {Preprint submitted to}
  {To appear in the}
  {}{}
\makeatother

\begin{frontmatter}

\title{Codensity, profiniteness and algebras of semiring-valued measures}

\author{Luca Reggio}

\address{IRIF, Universit\'e Paris Diderot, Sorbonne Paris Cit\'e, Case 7014, 75205 Paris Cedex 13, France and \\
  Laboratoire J.\ A.\ Dieudonn\'e, Universit\'e C{\^o}te d'Azur, Parc Valrose, 06108 Nice Cedex 02, France \\
  E-mail address: \href{mailto:reggio@unice.fr}{\textnormal{\texttt{reggio@unice.fr}}}\vspace{-2pt}}
  
  \journal{Journal of Pure and Applied Algebra}

\begin{abstract}
We show that, if $S$ is a finite semiring, then the free profinite $S$-semimodule on a Boolean Stone space $X$ is isomorphic to the algebra of all $S$-valued measures on $X$, which are finitely additive maps from the Boolean algebra of clopens of $X$ to $S$. These algebras naturally appear in the logic approach to formal languages as well as in idempotent analysis.
Whenever $S$ is a (pro)finite idempotent semiring, the $S$-valued measures are all given uniquely by continuous density functions. This generalises the classical representation of the Vietoris hyperspace of a Boolean Stone space in terms of continuous functions into the Sierpi{\' n}ski space. 

We adopt a categorical approach to profinite algebra which is based on profinite monads. The latter were first introduced by Ad{\'{a}}mek \emph{et al}.\ as a special case of the notion of codensity monads.
\end{abstract}


\end{frontmatter}

\section{Introduction}\label{s:intro}
Semirings generalise rings by relaxing the conditions on the additive structure requiring just a monoid rather than a group. The analogue of the notion of module over a ring is that of \emph{semimodule} over a semiring, or more concisely of an $S$-semimodule where  $S$ is the semiring. 
A \emph{profinite} $S$-semimodule is one which is isomorphic to the inverse limit (or cofiltered limit) of finite $S$-semimodules. Every profinite $S$-semimodule carries a topology turning it into a \emph{Boolean} (\emph{Stone}) \emph{space}, that is, a compact Hausdorff space admitting a basis of \emph{clopens}, i.e.\ of subsets which are simultaneously closed and open. By Stone duality for Boolean algebras \cite{Stone1936}, a Boolean space is completely determined by its Boolean algebra of clopen subsets equipped with the set-theoretic operations. 
We show in our main result, Theorem \ref{t:main-S-finite}, that the free profinite $S$-semimodule on a Boolean space $X$ is isomorphic to the algebra of all the measures on $X$ taking values in $S$, provided $S$ is finite. Here, the measurable subsets of $X$ are the clopens, and the measures on $X$ are only required to be finitely additive (cf.\ Definition \ref{d:measure}).

The motivation for the present work comes from logic, and more precisely from the connection between logic on words and the theory of formal languages. 
Many interesting classes of formal languages, both in the computational complexity hierarchy and within regular languages, correspond to fragments of logic. In this setting, certain quantifiers can be modelled by means of semirings (see, e.g., \cite{MP1971, STRAUBINGTT}). It is in understanding the effect of applying a layer of quantifiers to Boolean algebras of formulas, that profinite semimodules play an important r{\^ o}le. Further, in order to show that a certain construction is `optimal' from the viewpoint of language recognition (this can be thought of as a sort of `semantic completeness' result for a logic) it is crucial to have a \emph{concrete} description of these profinite objects, cf.\ \cite{GehrkePR16,GPR2017}. Such a description is provided by our main result.

Our measure-theoretic representation provides a bridge between several topics of interest. Firstly, it connects measures and profinite algebras. In this respect, it is related to Leinster's observation that the notions of integration and of codensity monads are tightly related \cite{Leinster2010}. Codensity monads are a special case of the concept of right Kan extension. Leinster's observation is that sometimes they can be seen as providing a correspondence akin to the one between `integration operators' and `measures'. This analogy becomes concrete in our measure-theoretic representation. Indeed, profinite algebras arise from a special class of codensity monads (see Section \ref{s:codensity-profinite-monads}), and we isolate a class of such monads admitting a concrete representation in terms of genuine measures. On the other hand, our main result makes a connection between measures and logic, as outlined above. Similar connections already exist and have proved useful. For example, in model theory, \emph{Keisler measures} \cite{Keisler1987} are probability measures on the Boolean algebras of definable subsets of models, and generalise the notion of (complete) types. Finally, we connect measures and semirings in the form of integration theory with coefficients in a semiring, which is the main focus of \emph{idempotent analysis} \cite{KM1997}. In the particular case of the tropical semiring, see Example \ref{ex:tropical-semiring}, this leads to \emph{tropical geometry}.

\medskip
In dealing with profinite algebras, we adopt a categorical approach. While monads allow for a categorical treatment of algebra, profinite algebra can be studied by means of \emph{profinite monads} \cite{Bojanczyk15, ChenAMU16}, a special case of right Kan extensions. In Section \ref{s:codensity-profinite-monads} we give a complete account of the basic theory of profinite monads meant to be accessible to non-experts in category theory. In particular, we show in Proposition \ref{p:hat-profinite-algebra} that profinite monads yield the expected notion of profinite algebra for varieties of Birkhoff algebras. This covers the case of the profinite $S$-semimodules free on Boolean spaces, for any $S$ (Corollary \ref{cor:free-prof-semim}).

We only obtain our measure-theoretic characterisation of the free profinite $S$-semimodule on a Boolean space for \emph{finite} semirings $S$, but in Section \ref{s:profinite-idempotent} we study the more general situation where $S$ is profinite. 
We show that the algebras of all $S$-valued measures satisfy a universal property with respect to those semimodules in which the scalar multiplication of $S$ is \emph{jointly continuous}, that we call `strongly continuous' semimodules (cf.\ Theorem \ref{t:free-on-idempotent-profinite-rig}). The case of finite semirings is considered in Section \ref{s:finite-case}. If $S$ is finite then every profinite $S$-semimodule is strongly continuous. Thus, we obtain our main result, Theorem \ref{t:main-S-finite}. 

Finally, in Section \ref{s:idempotent-case} we consider the case where $S$ is profinite and idempotent. In this setting, every measure is uniquely determined by its density function (see Theorem \ref{t:continuous-function-representation-1}). Provided $S$ is finite and idempotent, this yields a characterisation of the free profinite $S$-semimodule on a Boolean space $X$ as the algebra of all the continuous $S$-valued functions on $X$, with respect to an appropriate topology on $S$. 

The main result of this paper was announced in \cite{GPR2017}, where many details of the proofs were omitted. Here, we contribute a complete account of the topic, and we consider the main result from a wider perspective by studying algebras of measures with values in profinite semirings.

\section{Codensity and profinite monads}\label{s:codensity-profinite-monads}
The purpose of this section is to introduce the notion of a profinite monad. This is a special case of a more general construction, namely that of a codensity monad (which, in turn, is a special case of right Kan extension). Profinite monads provide a way of associating with a monad $T$ on the category of sets a monad $\p{T}$ on the category of Boolean spaces. We show in Proposition \ref{p:hat-profinite-algebra} that, whenever the monad $T$ is finitary, $\p{T}{X}$ is the free profinite $T$-algebra on the Boolean space $X$.
Although its content is categorical in nature, and the reader is supposed to be familiar with the basics of category theory, the section is written so as to be accessible to non-experts in category theory. In particular, we provide an elementary exposition of the notions involved up to the concept of monad as a categorical approach to algebra.
For a more thorough introduction to the theory of codensity monads we refer the interested reader to \cite{Leinster2010}. 
\subsection{Codensity monads: a brief introduction}\label{s:codensity}
We start by introducing a class of finitary monads on $\Set$ that will play a crucial r{\^o}le in the following, namely the \emph{semiring monads}. First, let us recall the following notions from algebra.
\begin{definition}\label{d:semiring-and-semimodule}
A \emph{semiring} is a tuple $S=(S,+,\cdot,0,1)$ such that $(S,+,0)$ is an Abelian monoid, $(S,\cdot,1)$ is a monoid, and for all $s,t,u\in S$ the laws
\begin{gather*}
s\cdot (t+u)= (s\cdot t) + (s\cdot u), \\
(t+u)\cdot s= (t\cdot s) + (u\cdot s), \\
s\cdot 0= 0= 0\cdot s
\end{gather*}
are satisfied. A \emph{semimodule over $S$} (or \emph{$S$-semimodule}, for short) is an Abelian monoid $M=(M,+_M,0_M)$ equipped with a `scalar multiplication' of $S$, that is, a function $S\times M\to M$, $(s,m)\mapsto sm$ satisfying 
\begin{gather*}
s(m+_M n)=sm +_M sn, \\
(s+t)m=sm +_M tm, \\
(s\cdot t)m=s(tm), \\
1m=m, \\
0m=0_M=s0_M
\end{gather*}
for all $s,t\in S$ and $m,n\in M$.
\end{definition}
\begin{example}\label{ex:semiring-monad}
Semimodules over semirings can be obtained as algebras for certain monads on $\Set$, called \emph{semiring monads}. 
Indeed, every semiring $S$ gives rise to a functor $\S\colon \Set\to\Set$ which associates with a set $X$ the set of finitely supported $S$-valued functions on $X$, i.e.\
\begin{align}\label{eq:finite-supp-functions}
\S{X}=\{f\colon X\to S\mid f(x)=0 \ \text{for all but finitely many} \ x\in X\}.
\end{align}
If $\phi\colon X\to Y$ is any function between sets, we get a function $\S{\phi}\colon\S{X}\to\S{Y}$ by setting 
\begin{align*}
\S{\phi}\colon f\mapsto \bigg(y\mapsto \sum_{\phi(x)=y}{f(x)}\bigg).
\end{align*}
Each $f\in\S{X}$ can be represented as a formal sum $\sum_{i=1}^n{s_i x_i}$, where $\{x_1,\ldots,x_n\}\subseteq X$ is the support of $f$, and $f(x_i)=s_i$ for each $i$. With this notation, we have $\S{\phi}(\sum_{i=1}^n{s_i x_i})=\sum_{i=1}^n{s_i \phi(x_i)}$. It is straightforward to check that $\S\colon\Set\to\Set$ is a functor. In fact, it is part of a monad $(\S,\eta,\mu)$ whose unit is
\[
\eta_X\colon X\to \S{X}, \ \eta_X(x)\colon x'\mapsto \begin{cases} 1 & \mbox{if } x'=x \\ 0 & \mbox{otherwise} \end{cases}
\]
(in other words, $\eta_X(x)$ is the $S$-valued characteristic function of $\{x\}$), and whose multiplication is
\[
\mu_X\colon \S^2{X}\to \S{X}, \ \sum_{i=1}^n{s_i f_i}\mapsto \bigg(x\mapsto\sum_{i=1}^n{s_i f_i(x)}\bigg).
\]
We remark that the only place where the multiplication of $S$ plays a r{\^ o}le is in the definition of the multiplication $\mu$ of the monad. We refer to $\S$ as the \emph{semiring monad} associated with $S$. Note that the finite power-set functor $\Pfin\colon\Set\to\Set$ is the semiring monad associated with the two-element distributive lattice 
\begin{equation*}
\2=(\{0,1\},\vee,\wedge,0,1),
\end{equation*}
regarded as a semiring. The algebras for the monad $\S$ are precisely the $S$-semimodules.
For example, if $S$ is $\2$, $\N=(\N,+,\cdot,0,1)$ or $\Z=(\Z,+,\cdot,0,1)$, then the algebras for $\S$ are semilattices, Abelian monoids and Abelian groups, respectively.
\end{example}

We briefly recall some basic facts from the theory of monads, see e.g.\ \cite{Joy1990}. If $T=(T,\eta,\mu)$ is a monad on a category $\mathbb{C}$, we write $\Alg{\mathbb{C}}{T}$ for the category of \emph{\textup{(}Eilenberg-Moore\textup{)} algebras for $T$}. In the special case where $T$ is a monad with rank and $\mathbb{C}$ is the category $\Set$ of sets and functions, the categories of the form $\Alg{\Set}{T}$ are, up to equivalence, exactly the varieties of algebras (with operations of possibly infinite, but bounded, arity). This correspondence restricts to categories of algebras for \emph{finitary} $\Set$-based monads (i.e., monads preserving filtered colimits) and varieties of Birkhoff algebras. See, e.g., \cite[{}VI.24]{Joy1990}. 

An interesting example of a $\Set$-monad which is not finitary is the ultrafilter monad. Write $\BStone$ for the category of Boolean spaces (i.e., compact Hausdorff spaces with a basis of clopens) and continuous maps. The underlying-set functor \[|-|\colon \BStone \to \Set\] has a left adjoint 
\begin{align}\label{eq:Stone-Cech-functor}
\beta\colon \Set\to \BStone
\end{align}
which sends a set $X$, regarded as a discrete space, to its Stone-{\v C}ech compactification $\beta X$. This adjunction induces a monad on $\Set$, the \emph{ultrafilter monad}, which is not finitary. By a theorem of Manes (see \cite[Section 1.5]{Manes1976} for a detailed exposition), its algebras are precisely the compact Hausdorff spaces. 

Whether $T$ is finitary or not, the category $\Alg{\Set}{T}$ is always equipped with a (regular epi, mono) factorisation system. In the category of compact Hausdorff spaces, this is the factorisation of a continuous map into a continuous surjection followed by a continuous injection. If $T$ is finitary, we recover the usual decomposition of a homomorphism of Birkhoff algebras into a surjective homomorphism followed by an injective one.

Codensity monads allow us to assign to (almost) any functor a monad on its codomain, and are a special case of a more general construction, namely that of \emph{right Kan extension}.
Henceforth, if $F$ and $G$ are any two parallel functors, $F\nat G$ denotes a natural transformation from $F$ to $G$. 
\begin{definition}\label{def:right-kan-extension}
Let $F\colon {\mathbb C}\to{\mathbb D}$ and $G\colon{\mathbb C}\to{\mathbb E}$ be any two functors. The \emph{right Kan extension of $F$ along $G$} is a pair $(K,\kappa)$, where $K\colon {\mathbb E}\to {\mathbb D}$ and $\kappa\colon K\circ G\nat F$, such that the following universal property is satisfied: for every pair $(K',\kappa')$ with $K'\colon {\mathbb E}\to {\mathbb D}$ and $\kappa'\colon K'\circ G\nat F$, there exists a unique natural transformation $\epsilon\colon K'\nat K$ such that the right-hand diagram below commutes.
If $F=G$, the right Kan extension of $G$ along itself is called the \emph{codensity monad of $G$}, and is denoted by $\K{G}$.
\begin{equation*}
\begin{tikzcd}
{\mathbb C} \arrow{dd}[swap]{G} \arrow[""{name=U, below}]{rr}{F} & & {\mathbb D} & & K'\circ G \arrow[Rightarrow]{rr}{\kappa'} \arrow[Rightarrow]{ddr}[swap]{\epsilon G} & & F \\
 & & & & & & &  \\
{\mathbb E} \arrow[Rightarrow, to=U, shorten <= 0.9em, shorten >= 0.2em, "\kappa"] \arrow{uurr}[swap, near start, inner sep=0.1ex, ""{name=V, above}]{K} \arrow[bend right=30, xshift=10pt, yshift=-10pt]{uurr}[swap, near start, inner sep=0.1ex, ""{name=W, above}]{K'}  \arrow[xshift=5pt, yshift=5pt, Rightarrow, from=W, to=V, shorten <= -0.3em, shorten >= -0.1em, "\epsilon", swap, dashed] & & & & & K\circ G \arrow[Rightarrow]{uur}[swap]{\kappa} &
\end{tikzcd}
\end{equation*}
\end{definition}
We remark that the fact that $\K{G}$ is a monad, i.e.\ it can be equipped with a unit and a multiplication, is a consequence of the universal property of the right Kan extension. Indeed, the unit of $\K{G}$ is obtained by taking $K'$ the identity functor and $\kappa'$ the identity natural transformation, while the multiplication is obtained by setting $K'=\K{G}\circ\K{G}$ and $\kappa'=\kappa\circ\K{G}\kappa$.

The right Kan extension of a functor along another one does not exist in general; however, it does exist under mild assumptions on the categories at hand, and can be computed as a limit. We state this precisely in the next lemma, in the special case of codensity monads. To do so, we first need to recall the notion of comma category. For a functor $G\colon {\mathbb C}\to{\mathbb D}$ and an object $d$ of ${\mathbb D}$, the \emph{comma category} $d\downarrow G$ has as objects pairs $(\alpha, c)$, where $c$ is an object of ${\mathbb C}$ and $\alpha\colon d\to Gc$ is a morphism in ${\mathbb D}$. A morphism between two objects $(\alpha_1,c_1),(\alpha_2,c_2)$ of $d\downarrow G$ is a morphism $f\colon c_1\to c_2$ in ${\mathbb C}$ such that $Gf\circ \alpha_1=\alpha_2$.
\begin{lemma}[{\cite[Theorem X.3.1]{MacLane}}]\label{l:limit-formula}
Let $G\colon {\mathbb C}\to{\mathbb D}$ be any functor. If ${\mathbb C}$ is essentially small \textup{(}i.e., equivalent to a small category\textup{)} and ${\mathbb D}$ is complete, then the codensity monad $\K{G}\colon {\mathbb D}\to {\mathbb D}$ exists and for every $d$ in ${\mathbb D}$
\begin{align*}
\K{G}d=\lim_{d\to Gc} {Gc},
\end{align*}
where the limit is taken over the comma category $d\downarrow G$.\qed
\end{lemma}

An example of codensity monad is provided by a result of Kennison and Gildenhuys \cite{KENNISON1971317}, which identifies the codensity monad of the inclusion $\Setfin\to\Set$ of finite sets into sets as the ultrafilter monad on $\Set$. Recently, Leinster \cite{Leinster2010} has reinterpreted Kennison and Gildenhuys' result as a correspondence between measures (i.e.\ ultrafilters, or two-valued measures) and integration operators (i.e.\ elements of the free algebras for the codensity monad). He then takes the analogy further to study the codensity monad of the inclusion of finite-dimensional vector spaces into the category of vector spaces \cite[Section 7]{Leinster2010}. Proposition \ref{p:hat-profinite-algebra} and Theorem \ref{t:main-S-finite} together identify a class of codensity monads whose (free) algebras admit a description as algebras of \emph{bona fide} measures, thus providing a setting in which the analogy above is concretely realised.

\begin{example}\label{ex:powerset-vietoris}
Consider the finite power-set monad $\Pfin$ on the category of sets. The algebras for this monad are semilattices (cf.\ Example \ref{ex:semiring-monad}), thus the finitely carried algebras are the finite semilattices. Let $G$ be the functor from the category of finite semilattices and semilattice homomorphisms, to the category $\BStone$ of Boolean spaces, which regards the underlying set of a finite semilattice as a discrete Boolean space. The codensity monad $\K{G}$ of this functor is the \emph{Vietoris monad} on $\BStone$. 
Although this fact can be proved directly, it will follow at once from equation \eqref{eq:limit-formula-locally-finite} below, along with the fact that the Vietoris functor on $\BStone$ preserves codirected limits \cite[{}3.12.27(f)]{Engelking}.
We briefly recall how the Vietoris monad is defined. If $X$ is a Boolean space, write $\V{X}$ for the set of closed subsets of $X$, and equip it with the topology generated by the sets of the form
\[
\boxa C=\{V\in\V{X}\mid V\subseteq C\} \ \ \text{and} \ \ \Diamond C=\{V\in\V{X}\mid V\cap C\neq\emptyset\}
\] 
for $C$ a clopen of $X$. The ensuing topological space is called the \emph{Vietoris hyperspace of $X$}, and is again a Boolean space \cite[Theorem 4.9]{Michael1951}. This is a particular case of the hyperspace of an arbitrary topological space first introduced by Vietoris in 1922, see \cite{Vietoris1923}. For any continuous function $f\colon X\to Y$ between Boolean spaces, defining $\V{f}\colon \V{X}\to\V{Y}$ as the forward image function yields a functor $\V\colon \BStone\to\BStone$. The unit of the Vietoris monad is
\begin{equation*}
\eta_X\colon X\to \V{X}, \ x\mapsto \{x\},
\end{equation*}
and the multiplication is
\begin{equation*}
\mu_X\colon \V^2{X}\to\V{X}, \ S\mapsto \bigcup_{V\in S}{V}.
\end{equation*}
Note that the definition of the components $\mu_X$ goes back at least to \cite[Theorem 5 p.\ 52]{Kuratowski2}.
\end{example}

\subsection{Profinite monads and their algebras}\label{s:profinite-monads}
Profinite monads allow us to associate with every monad $T=(T,\eta,\mu)$ on the category of sets a monad $\p{T}=(\p{T},\w{\eta},\w{\mu})$ on the category of Boolean spaces, called the \emph{profinite monad} of $T$. If $\Algfin{\Set}{T}$ denotes the full subcategory of $\Alg{\Set}{T}$ on the finitely carried $T$-algebras, we can consider the composition
\begin{equation*}
\G\colon \Algfin{\Set}{T} \to \Setfin \to \BStone
\end{equation*}
of the underlying-set functor from finite $T$-algebras to finite sets, followed by the full embedding of finite sets into the category of Boolean spaces. Note that $\Algfin{\Set}{T}$ is essentially small and $\BStone$ is complete, so that Lemma \ref{l:limit-formula} applies to $\G$. The profinite monad $\p{T}$ is defined as the codensity monad of the functor $G$. That is
\begin{equation}\label{eq:profinite-monad}
\p{T}=\K{\G}\colon \BStone \to\BStone, \ \ \p{T}X=\lim_{X\to G(Y,h)}{G(Y,h)}.
\end{equation}
\begin{remark}\label{r:profinite-monads-language-theory}
Profinite monads were first introduced as a natural categorical extension of the profinite algebraic methods which are heavily used in the theory of regular languages. In \cite{Bojanczyk15}, Boja{\'n}czyk associates with a $\Set$-monad $T$ another $\Set$-monad which models the profinite version of the objects modelled by $T$. Profinite monads, as defined above, first appeared in \cite{ChenAMU16}, where it is pointed out that Boja{\'n}czyk's construction can be recovered by composing the monad $\p{T}$ with the adjunction $\beta\colon\Set\leftrightarrows \BStone\cocolon |-|$ in \eqref{eq:Stone-Cech-functor}. 
We point out that in \cite{ChenAMU16} the authors consider, more generally, monads on a variety of Birkhoff algebras $\mathscr{V}$. The associated profinite monad is then a monad on the category of profinite $\mathscr{V}$-algebras. Here, we shall deal only with the case where $\mathscr{V}=\Set$.
The monadic approach to formal language theory, put forward by Boja{\'n}czyk in \emph{op.\ cit.}, is also adopted in \cite{GPR2017} to deal with different kinds of quantifiers at the same time. Since our measure-theoretic representation of the free profinite $S$-semimodules is instrumental to some of the main results in \cite{GPR2017}, in the following we give a complete account of the basic theory of profinite monads.
\end{remark}

The Vietoris hyperspace monad $\V$ on Boolean spaces, which coincides with the profinite monad of the finite power-set monad $\Pfin$ on the category of sets (see Example \ref{ex:powerset-vietoris}), provides a prime example of the profinite monad construction. The original monad $\Pfin$ and its profinite extension $\V$  come equipped with a `comparison map': for each Boolean space $X$ there is a function $\t_X\colon \Pfin{X}\to \V{X}$ which views a finite subset of the space $X$ as a closed subspace.
The ensuing natural transformation $\t$ plays a key r{\^o}le, and can be defined for any profinite monad, as we shall now explain.
Write 
\begin{equation}\label{eq:kappa-universal-nt}
\kappa\colon \p{T}\circ G\nat G
\end{equation}
 for the natural transformation such that the pair $(\p{T},\kappa)$ satisfies the universal property defining the right Kan extension. The forgetful functor $|-|\colon\BStone\to\Set$ is right adjoint, hence it commutes with right Kan extensions \cite[Theorem X.5.1]{MacLane}. That is, $|-|\circ\p{T}$ is the right Kan extension of $|-|\circ G$ along $G$. Now, consider the left-hand diagram below.
\begin{equation*}
\begin{tikzcd}
\Algfin{\Set}{T} \arrow{dd}[swap]{G} \arrow[""{name=U, below}]{rr}{|-|\circ G} & & \Set & & \BStone \arrow{rr}{\p{T}} \arrow{dd}[swap]{|-|} & & \BStone \arrow{dd}{|-|} \\
 & & & & & & &  \\
\BStone \arrow[Rightarrow, to=U, shorten <= 1.6em, shorten >= 1.2em, "|-|\kappa"] \arrow{uurr}[swap, near start, inner sep=0.1ex, ""{name=V, above}]{|-|\circ\p{T}} \arrow[dashed, bend right=30, xshift=10pt, yshift=-2pt]{uurr}[swap, inner sep=0.1ex, ""{name=W, above}]{T\circ|-|}  \arrow[dashed, xshift=5pt, xshift=0pt, yshift=6pt, Rightarrow, from=W, to=V, shorten <= 0.6em, shorten >= -0.1em, "\t", swap] & & & & \Set \arrow{rr}[swap]{T} \arrow[Rightarrow, shorten <= 2.0em, shorten >= 2.0em]{uurr}{\t} & & \Set
\end{tikzcd}
\end{equation*}
There is an obvious natural transformation $\alpha\colon T\circ|-|\circ G\nat |-|\circ G$ whose component at a finite $T$-algebra $(X,h)$ is simply $\alpha_{(X,h)}= h$. Therefore, by the universal property of the right Kan extension $(|-|\circ\p{T},|-|\kappa)$, there is a unique natural transformation $\t\colon T\circ|-|\nat |-|\circ \p{T}$ as in the right-hand diagram above, satisfying 
\begin{equation}\label{eq:property-defining-tau}
|-|\kappa\circ \t G=\alpha.
\end{equation}
In view of equation \eqref{eq:profinite-monad}, the components of the natural transformation $\t$ admit explicit descriptions as limit maps. Since the functor $|-|$ preserves limits, we have
\[
|\p{T}X|=|\lim_{X\to G(Y,h)}{G(Y,h)}|=\lim_{X\to G(Y,h)}{Y}.
\]
In turn, each object $(X\xrightarrow{\phi} G(Y,h),(Y,h))$ of the comma category $X\downarrow G$ yields a function $\phi^*\colon T|X|\to Y$ given by $\phi^*= h\circ T|\phi|$. Note that $\phi^*$ is the unique $T$-algebra morphism extending the function $|\phi|\colon |X|\to Y$. 
\begin{definition}\label{d:nt-tau}
Let $T$ be any monad on $\Set$, and $\p{T}$ its profinite monad. We define $\t\colon T\circ|-|\nat |-|\circ \p{T}$ to be the unique natural transformation satisfying \eqref{eq:property-defining-tau}. For any $X$ in $\BStone$, the component $\t_X\colon T|X|\to|\p{T}X|$ is the unique function induced by the cone 
\[\{\phi^*\colon T|X|\to Y\mid (\phi,(Y,h))\in X\downarrow G\}.\]
\end{definition}

In \cite[Proposition B.7.(a)]{ACMU2015arxiv} the authors prove that the natural transformation $\t$ behaves well with respect to the units and multiplications of the monads $T$ and $\p{T}$. That is, in the terminology of \cite{Street}, $(|-|,\t)\colon (\BStone,\p{T})\to(\Set,T)$ is a \emph{monad functor}. This means that the next two diagrams commute.
\begin{equation}\label{eq:tau-morphism-of-monads}
\begin{tikzcd}
{|-|} \arrow[Rightarrow]{rr}{|-|\w{\eta}} \arrow[Rightarrow]{dr}[swap]{\eta|-|} & & {|-|\circ\p{T}} & & {T^2\circ|-|} \arrow[Rightarrow]{d}[swap]{T\t} \arrow[Rightarrow]{rr}{\mu|-|} & & {T\circ|-|} \arrow[Rightarrow]{d}{\t} \\
{} & {T\circ|-|} \arrow[Rightarrow]{ur}[swap]{\t} & & & {T\circ|-|\circ\p{T}} \arrow[Rightarrow]{r}{\t\p{T}} & {|-|\circ\p{T}^2} \arrow[Rightarrow]{r}{|-|\w{\mu}} & {|-|\circ\p{T}}
\end{tikzcd}
\end{equation}
An immediate consequence of this is that the underlying-set functor $|-|\colon\BStone\to\Set$ lifts to a functor $\Alg{\BStone}{\p{T}}\to\Alg{\Set}{T}$, thus showing that every algebra for the profinite monad $\p{T}$ admits a $T$-algebra reduct, and every morphism of $\p{T}$-algebras preserves this structure. For a free $\p{T}$-algebra $\p{T}X$, its $T$-algebra reduct is provided by the composition
\begin{equation}\label{eq:T-alg-structure-on-T-hat}
\begin{tikzcd}
{T|\p{T}X|} \arrow{r}{\t_{\p{T}X}} & {|\p{T}^2 X|} \arrow{r}{|\w{\mu}_X|} & {|\p{T}X|}.
\end{tikzcd}
\end{equation}
\begin{lemma}\label{l:T-alg-structure-tau-morphism}
For every $\Set$-monad $T$ and Boolean space $X$, the map in \eqref{eq:T-alg-structure-on-T-hat} yields a $T$-algebra structure on \textup{(}the underlying set of\textup{)} the space $\p{T}{X}$ such that the map $\t_X\colon T|X|\to |\p{T}X|$ is a morphism of $T$-algebras.
\end{lemma}
\begin{proof}
This is a direct consequence of the commutativity of the diagrams in \eqref{eq:tau-morphism-of-monads}.
\end{proof}
In the case of the map $\t_X\colon\Pfin{X}\to\V{X}$ the previous lemma states that the Vietoris space $\V{X}$ is a semilattice when equipped with the binary operation $\cup$, and $\t_X\colon (\Pfin{X},\cup)\to(\V{X},\cup)$ is a semilattice homomorphism.
Another important property of the map $\t_X\colon\Pfin{X}\to\V{X}$ is the well-known fact (see \cite[Theorem 4 p.\ 163]{Kuratowski1}) that it has dense image. In fact, this feature is common to all profinite monads. In the special case of a finite discrete space $X$, this follows from \cite[Proposition B.7.(b)]{ACMU2015arxiv}.
\begin{lemma}\label{l:dense-image}
For every $\Set$-monad $T$ and Boolean space $X$, the map $\t_X\colon T|X|\to |\p{T}X|$ has dense image.
\end{lemma}
\begin{proof}
Since the category $X\downarrow G$ is codirected, it is enough to show that every non-empty subbasic open set of $\p{T}X$, in the limit topology, contains an element in the image of $\t_X$. Such an open set is of the form $p^{-1}(y)$, where $p\colon \p{T}X\to Y$ is a continuous function in the limit cone defining $\p{T}X$ (cf.\ equation \eqref{eq:profinite-monad}) and $y\in Y$ is in the image of $p$.
More precisely, this means that there exists $(X\xrightarrow{\phi} G(Y,h),(Y,h))$ in the comma category $X\downarrow G$ such that $|p|\circ \t_X=\phi^*\colon T|X|\to Y$. To settle the statement, it thus suffices to prove $(\phi^*)^{-1}(y)\neq\emptyset$.

Recall that $\phi^*$ is the $T$-algebra morphism obtained as the free extension of the function $|\phi|\colon |X|\to Y$. We can then consider its (regular epi, mono) factorisation in the category of $T$-algebras as displayed below.
\[\begin{tikzcd}
(T|X|,\mu_{|X|}) \arrow{rr}{\phi^*} \arrow[twoheadrightarrow]{dr}[swap]{e} & & (Y,h) \\
 & (Y',h') \arrow[rightarrowtail]{ur}[swap]{m} & 
\end{tikzcd}\]
The map $e\colon T|X|\to Y'$ is surjective, hence it is enough to prove $m^{-1}(y)\neq \emptyset$. Note that $m$ is a morphism in the category $X\downarrow G$. Indeed, $e\circ \eta_{|X|}\colon |X|\to Y'$ is the underlying function of a continuous map $\phi'\colon X\to Y'$ (namely, an appropriate corestriction of $\phi$) and $m\circ \phi'=\phi$. Hence, $m\colon (\phi',(Y',h'))\to(\phi,(Y,h))$ is a morphism in $X\downarrow G$. It follows that there exists $p'\colon |\p{T}X|\to Y'$ satisfying $m\circ p'=p$. Since $y$ is in the image of $p$ by hypothesis, it is also in the image of $m$, as was to be shown.
\end{proof}
In general, the morphisms $\t_X\colon T|X|\to |\p{T}X|$ do not have to be injective. A counterexample is provided by the power-set monad $\P$ on $\Set$, whose profinite monad is again the Vietoris monad. In this case the map $\t_X\colon\P{X}\to\V{X}$ sends a subset of $X$ to its topological closure, and it is injective precisely when $X$ is finite. 

However, the components of the natural transformation $\t$ are injective provided the monad $T$ is finitary and restricts to finite sets. To see this observe that, whenever $T$ restricts to finite sets, the underlying-set functor $\Algfin{\Set}{T}\to\Setfin$ is right adjoint and is thus preserved by right Kan extensions. It follows that the limit formula in \eqref{eq:profinite-monad} can be considerably simplified to yield, for every Boolean space $X$,
\begin{equation}\label{eq:limit-formula-locally-finite}
\p{T}{X}=\lim_{X\epifin Y}{TY}.
\end{equation}
Here, the notation $X\epifin Y$ means that $Y$ is a finite continuous image of $X$. Moreover, the limit is computed in $\BStone$ by equipping the finite sets $T{Y}$ with the discrete topology. In this setting the function $\t_X$ is the limit map for the cone
\begin{equation*}
\{T|\phi|\colon T|X|\to TY\mid \phi\colon X\epifin Y\}
\end{equation*}
and hence it is injective if, and only if, this cone is jointly monic. Suppose $f,g\colon S\to T|X|$ are any two functions, and $f(s)\neq g(s)$ for some $s\in S$. If $T$ is finitary, and $\mathcal{F}$ is the collection of finite subsets of $|X|$,
\begin{equation*}
T|X|=T\big(\colim_{F\in\mathcal{F}}{F}\big)=\colim_{F\in\mathcal{F}}{TF}
\end{equation*}
implies the existence of a finite subset $F$ of $X$ such that $f(s),g(s)\in TF$. Since $X$ is a Boolean space, there is a finite discrete space $Z$ and a continuous surjection $\psi\colon X\epifin Z$ such that $\psi$ separates any two distinct elements of $F$. Then $T|\psi|$ distinguishes $f(s)$ and $g(s)$, showing that the cone is jointly monic.

Regarding the injectivity of $\t_X$, it will follow from Proposition \ref{p:hat-profinite-algebra} below that the hypothesis that $T$ restricts to finite sets cannot, in general, be dropped. Indeed, for a finitary monad $T$ and a finite discrete space $X$, injectivity of $\t_X$ corresponds to the free finitely generated $T$-algebra $T|X|$ being residually finite. That is, any two distinct morphisms into $T|X|$ can be separated by a morphism $T|X|\to Y$ with $Y$ a finitely carried $T$-algebra.
In turn, Birkhoff varieties containing no non-trivial finite members (see, e.g., \cite{Austin1965}) yield obvious examples of monads $T$ for which $\t_X$ fails to be injective.

We conclude the section by showing that, whenever the monad $T$ is finitary, $\p{T}X$ is a profinite $T$-algebra, i.e.\ it belongs to the pro-completion $\mathbf{Pro}$-$\Algfin{\Set}{T}$ of $\Algfin{\Set}{T}$ (cf.\ \cite[Chapter VI]{Johnstone1986}). The objects of $\mathbf{Pro}$-$\Algfin{\Set}{T}$ can be identified with those $T$-algebras $Y$ carrying Boolean topologies such that there exists a cone of continuous $T$-algebra morphisms
\[
\{Y\to Y_i\mid i\in I\},
\]
where $\{Y_i\mid i\in I\}$ is a codirected diagram of finite $T$-algebras equipped with the discrete topologies, whose image in $\BStone$ is a limit cone.\footnote{\label{foot:profinite}Let $\mathscr{V}$ be a variety of Birkhoff algebras, and $\mathscr{V}_f$ its full subcategory on the finite members. The pro-completion $\mathbf{Pro}$-$\mathscr{V}_f$ is equivalent to the full subcategory of the category $\BStone\mathscr{V}$, whose objects are topological $\mathscr{V}$-algebras based on Boolean spaces and whose morphisms are continuous homomorphisms, on those objects which are limits of codirected diagrams of finite discrete $\mathscr{V}$-algebras \cite[Corollary VI.2.4]{Johnstone1986}. Since the forgetful functor $\BStone\mathscr{V}\to \BStone$ preserves and reflects limits, this description of $\mathbf{Pro}$-$\mathscr{V}_f$ is equivalent to the one given above, where $T$ is the finitary monad associated with the variety $\mathscr{V}$.} Further, the morphisms in $\mathbf{Pro}$-$\Algfin{\Set}{T}$ can be identified with the continuous $T$-algebra morphisms.
 In fact, $\p{T}X$ is the free profinite $T$-algebra on $X$. That is, $\p{T}$ is the monad induced by the forgetful functor $\mathbf{Pro}$-$\Algfin{\Set}{T}\to\BStone$, and its left adjoint. We thus recover the folklore result stating that, for any Boolean space $X$, the Vietoris space $\V{X}$ is the free profinite semilattice on $X$.
\begin{proposition}\label{p:hat-profinite-algebra}
Let $T$ be a finitary $\Set$-monad, and $X$ a Boolean space. Then $\p{T}X$ is the free profinite $T$-algebra on the Boolean space $X$.
\end{proposition}
\begin{proof}
We first show that $\p{T}X$ is a profinite $T$-algebra. Let 
\[
\{\pi_Y\colon \p{T}X\to Y\mid (\phi,(Y,h))\in X\downarrow G\}
\]
be the cone of continuous functions defining $\p{T}X$ as an inverse limit. It suffices to show that each $|\pi_Y|$ is a $T$-algebra morphism.
In turn, this amounts to saying that the outer rectangle below commutes,
\[\begin{tikzcd}
{T|\p{T}X|} \arrow{r}{\t_{\p{T}X}} \arrow{dd}[swap]{T|\pi_Y|} & {|\p{T}^2{X}|} \arrow{r}{|\w{\mu}_X|} \arrow[dashed]{d}{|\p{T}\pi_Y|} & {|\p{T}X|} \arrow{dd}{|\pi_Y|} \\
 & {|\p{T}Y|} \arrow[dashed]{dr}[xshift=-5pt]{|\kappa_{(Y,h)}|} & \\
 TY \arrow[dashed]{ur}{\t_Y} \arrow{rr}{h} & & Y
\end{tikzcd}\]
where $\kappa\colon \p{T}\circ G\nat G$ is as in \eqref{eq:kappa-universal-nt}.
The bottom triangle commutes by \eqref{eq:property-defining-tau}, while the left-hand trapezoid commutes by naturality of $\t$.
Finally, the commutativity of the right-hand trapezoid follows from the equalities $\kappa\circ \p{T}\kappa=\kappa\circ\w{\mu}G$ and $\pi_Y=\kappa_{(Y,h)}\circ \p{T}\phi$. The first one is obtained by noticing that $\w{\mu}\colon \p{T}^2\nat \p{T}$ is the unique natural transformation induced by the universal property of the right Kan extension $(\p{T},\kappa)$ and the natural transformation $\kappa\circ\p{T}\kappa\colon \p{T}^2\circ G\nat G$. 
For the second equality, it suffices to show that $|\pi_Y|\circ\t_X=|\kappa_{(Y,h)}\circ \p{T}\phi|\circ\t_X$. In turn, this follows from naturality of $\t$, and the fact that $|\pi_Y|\circ\t_X=h\circ T|\phi|$.

It remains to prove that $\p{T}X$ satisfies the universal property with respect to the unit $\w{\eta}_X\colon X\to \p{T}X$. That is, for every profinite $T$-algebra $Y$ and every continuous map $f\colon X\to Y$ there is a unique continuous morphism of $T$-algebras $\p{T}X\to Y$ making the following diagram commute.
\begin{equation}\label{eq:universal-property}
\begin{tikzcd}
X \arrow{r}{\w{\eta}_X} \arrow{dr}[swap]{f} & \p{T}X \arrow[dashed]{d} \\
& Y
\end{tikzcd}
\end{equation}
Note that, if such a map exists, it is unique. Indeed, assume $g_1,g_2\colon \p{T}X\to Y$ are continuous $T$-algebra morphisms making diagram \eqref{eq:universal-property} commute. By Lemma \ref{l:dense-image}, along with the fact that $\p{T}X$ is Hausdorff, if we prove $|g_1|\circ \t_X=|g_2|\circ \t_X$ it will follow that $g_1=g_2$. By the universal property of the free $T$-algebra $T|X|$ there is a unique $T$-algebra morphism $\xi\colon T|X|\to |Y|$ extending $|f|\colon |X|\to |Y|$, i.e.\ satisfying
\begin{equation}\label{eq:unique-T-algebra-extension}
\xi\circ |\eta_X|=|f|.
\end{equation}
By Lemma \ref{l:T-alg-structure-tau-morphism} the maps $|g_1|\circ \t_X, |g_2|\circ \t_X$ are $T$-algebra morphisms. In turn, the left-hand diagram in \eqref{eq:tau-morphism-of-monads} entails that they are both solutions to equation \eqref{eq:unique-T-algebra-extension}. We obtain that $|g_1|\circ \t_X=|g_2|\circ \t_X$, whence $g_1=g_2$.

To conclude, we prove that diagram \eqref{eq:universal-property} admits a solution. Let $\{Y_i\mid i\in D\}$ be a codirected diagram of finite discrete $T$-algebras, and $\{\rho_i\colon Y\to Y_i\mid i\in D\}$ a cone of continuous $T$-algebra morphisms whose image in $\BStone$  is a limit cone. Note that each $Y_i$ belongs to the diagram defining $\p{T}Y$ (see equation \eqref{eq:profinite-monad}). Thus, for every $i\in D$, there is a continuous map $\pi_i\colon \p{T}Y\to Y_i$ in the limit cone. As we saw in the first part of the proof, each $|\pi_i|$ is a $T$-algebra morphism. Consider now the continuous map $\phi_i\colon \p{T}X\to Y_i$ defined by $\phi_i=\pi_i\circ\p{T}f$. The function $|\p{T}f|$ is a $T$-algebra morphism, as pointed out before Lemma \ref{l:T-alg-structure-tau-morphism}. It follows that 
\[
\{\phi_i\colon \p{T}X\to Y_i\mid i\in D\}\]
is a cone of continuous $T$-algebra morphisms. Since $Y$ is the limit of $\{Y_i\mid i\in D\}$ in $\mathbf{Pro}$-$\Algfin{\Set}{T}$, there is a unique continuous $T$-algebra morphism $\phi\colon  \p{T}X\to Y$ such that $\rho_i\circ\phi=\phi_i$ for every $i\in D$. We claim that $\phi$ is a solution to \eqref{eq:universal-property}. It is enough to prove $\rho_i\circ \phi\circ\w{\eta}_X= \rho_i\circ f$ for every $i\in D$. In turn, this follows from the definition of $\phi$ and the commutativity of the left-hand diagram in \eqref{eq:tau-morphism-of-monads}.
\end{proof}

Recall from Example \ref{ex:semiring-monad} the notion of semiring monad on $\Set$. We conclude by specialising Proposition \ref{p:hat-profinite-algebra} to the particular case of a semiring monad $\S$, and its profinite monad $\p{\S}$ on $\BStone$.
\begin{corollary}\label{cor:free-prof-semim}
Let $S$ be any semiring, and $X$ a Boolean space. Then $\p{\S}X$ is the free profinite $S$-semimodule on the Boolean space $X$. \qed
\end{corollary}

\section{Algebras of measures with values in profinite semirings}\label{s:profinite-idempotent}
In the previous section we saw that, for any semiring $S$, the free profinite $S$-semimodule on a Boolean space $X$ is isomorphic to $\p{\S}X$ (Corollary \ref{cor:free-prof-semim}), where $\p{\S}$ is the profinite monad of the semiring  monad associated with $S$. We are interested in a concrete description of the profinite algebra $\p{\S}X$. It turns out that $\p{\S}X$ can be identified with the algebra of all the $S$-valued measures on $X$ (in the sense of Definition \ref{d:measure}), provided $S$ is \emph{finite}. This is the content of Section \ref{s:finite-case}. In the present section, we deal with the more general case of \emph{profinite} semirings $S$. Here, it is not true that the free profinite $S$-semimodule on a Boolean space $X$ is isomorphic to the algebra of all the $S$-valued measures on $X$ (cf.\ Remark \ref{r:quotient-of-strongly-cont}). However, we shall see in Theorem \ref{t:free-on-idempotent-profinite-rig} that the latter algebra enjoys a universal property relative to those profinite $S$-semimodules in which the scalar multiplication of $S$ is \emph{jointly continuous}. If $S$ is finite, then separate and joint continuity coincide, thus providing the desired measure-theoretic representation of $\p{\S}X$.

Throughout this section we work with a fixed profinite semiring $S=(S,+,\cdot,0,1)$, i.e.\ $S$ is the limit of an inverse system of finite discrete semirings and semiring homomorphisms. Every profinite semiring is a Boolean topological semiring.\footnote{This fact is not specific about semirings, and it holds for any variety of Birkhoff algebras. Cf.\ footnote \ref{foot:profinite}.} In view of \cite[Proposition 7.2]{SZ2017}, the converse is also true. That is, a semiring is profinite if, and only if, it is equipped with a Boolean topology which makes the operations $+$ and $\cdot$ continuous. Every finite semiring, endowed with the discrete topology, is trivially profinite. Two infinite profinite semirings are described in Examples \ref{ex:one-point-extended} and \ref{ex:tropical-semiring} below. 

Before proceeding, we briefly recall some basic facts about \emph{Stone duality} for Boolean algebras \cite{Stone1936} which will be used in the remainder of the paper.
Concisely, Stone duality states that the category $\Bool$ of Boolean algebras and their homomorphisms is dually equivalent to the category $\BStone$ of Boolean spaces and continuous maps. In more detail, the \emph{dual algebra} $B_X$ (or simply $B$, if no confusion arises) of a Boolean space $X$ is the Boolean algebra consisting of the clopen subsets of $X$, equipped with set-theoretic operations. For any continuous map $f\colon X\to Y$ between Boolean spaces, the inverse image function $f^{-1}\colon B_Y\to B_X$ is a Boolean algebra homomorphism. In the converse direction, the \emph{dual space} of a Boolean algebra $B$ is the Boolean space $X_B$ of ultrafilters of $B$. Recall that an \emph{ultrafilter} of $B$ is a proper subset $x\subseteq B$ which is closed under finite meets, upward closed, and for every $a\in B$ it satisfies either $a\in x$ or $\neg a\in x$. The set $X_B$ is equipped with the \emph{Stone topology} generated by the basis of clopens consisting of the sets of the form
\[
\widehat{a}=\{x\in X_B\mid a\in x\},
\] 
for $a\in B$. If $h\colon A\to B$ is a homomorphism of Boolean algebras, the function $h^{-1}\colon X_B\to X_A$ is continuous with respect to the Stone topologies. Stone duality for Boolean algebras states that, up to a natural isomorphism, the ensuing contravariant functors $\Bool\leftrightarrows\BStone$ are inverse to each other. In particular, the isomorphism $a\mapsto \widehat{a}$ allows one to recover the Boolean algebra $B$ as the algebra of clopens of $X_B$.
The connection between Stone duality and profinite algebra is a deep one, and it was fully exposed in \cite{Gehrke2016}.

\begin{example}\label{ex:one-point-extended}
Let $\Nt$ be the (Alexandroff) one-point compactification of the set $\N$ of natural numbers, regarded as a discrete space. That is, $\Nt=\N\cup\{\infty\}$ and its opens are the subsets which are either cofinite, or do not contain $\infty$. It is well-known that $\Nt$ is the dual space of the Boolean subalgebra of $\P(\N)$ consisting of the finite and cofinite subsets of $\N$.
The usual addition and multiplication on $\N$ can be extended to $\Nt$ by setting 
\[
\forall x\in\Nt, \ \ x+\infty=\infty \ \ \text{and} \ \ x\cdot \infty = \begin{cases} 0 & \mbox{if } x=0, \\ \infty & \mbox{otherwise.}  \end{cases}
\]
This gives a semiring $(\Nt,+,\cdot,0,1)$ which is easily seen to be topological, hence profinite. 
\end{example}
\begin{example}\label{ex:tropical-semiring}
We equip the Boolean space $\Nt$, defined in the previous example, with a different semiring structure. Define the addition of the semiring to be the $\min$ operation (with identity element $\infty$), and the multiplication to be $+$. The ensuing idempotent semiring $(\Nt,\min,+,\infty,0)$ is called the \emph{\textup{(}min-plus\textup{)} tropical semiring}. The operations $\min$ and $+$ are continuous with respect to the Boolean topology, hence this is a profinite semiring.
The tropical semiring plays an important r{\^o}le in the theory of formal languages, see e.g.\ the survey \cite{Pin1998}.
\end{example}
Next we introduce the notion of measure, that will play a central r{\^o}le throughout.
\begin{definition}\label{d:measure}
Let $X$ be a Boolean space with dual algebra $B$. An \emph{$S$-valued measure} (or simply a  \emph{measure}, if the semiring is clear from the context) on $X$ is a function $\mu\colon B\to S$ which is finitely additive, i.e.\
\begin{enumerate}
\item $\mu(0)=0$;
\item $\mu(a\vee b)=\mu(a)+\mu(b)$ whenever $a,b\in B$ satisfy $a\wedge b=0$.
\end{enumerate}
Item $2$ can be expressed without reference to disjointness, in the following way:
\[
\forall a,b\in B, \ \mu(a\vee b)+\mu(a\wedge b)=\mu(a)+\mu(b).
\]
\end{definition}
For any Boolean space $X$ with dual algebra $B$, write
\begin{align*}
\M(X,S)=\{\mu\colon B\to S\mid \mu \ \text{is a measure}\}
\end{align*}
for the set of all the $S$-valued measures on $X$. The latter is naturally equipped with a structure of $S$-semimodule, whose operations are computed pointwise:
\[
\forall s\in S \ \forall \mu_1,\mu_2\in\M(X,S), \ \ \mu_1+\mu_2\colon b\mapsto \mu_1(b)+\mu_2(b) \ \ \text{and} \ \ s\cdot \mu_1\colon b\mapsto s\cdot\mu_1(b).
\]

On the other hand, $\M(X,S)$ can also be equipped with a natural topology, namely the subspace topology induced by the product topology of $S^B$. This coincides with the initial topology for the set of \emph{evaluation functions}
\begin{align}\label{eq:evaluation-map}
\ev_b\colon \M(X,S)\to S, \ \ \mu\mapsto \mu(b),
\end{align}
for $b\in B$. Note that $\ev_b$ is the restriction of the $b$-th projection $S^B\to S$.
 A subbasis for this topology is given by the sets of the form 
\begin{align}\label{eq:basic-opens-measures}
\langle b,U\rangle=\{\mu\in\M(X,S)\mid \mu(b)\in U\},
\end{align}
for $b\in B$ and $U$ a clopen subset of $S$. With respect to this topology, $\M(X,S)$ is a Boolean space.

\begin{lemma}\label{l:measures-are-boolean}
For any Boolean space $X$, the space $\M(X,S)$ of all the $S$-valued measures on $X$ is Boolean.
\end{lemma}
\begin{proof}
By Tychonov's theorem, the product topology on $S^B$ is compact. Since $S$ admits a basis of clopens, so does $S^B$. Hence, $S^B$ is a Boolean space. Since a closed subspace of a Boolean space is Boolean, it is enough to prove that $\M(X,S)$ is a closed subset of $S^B$. If $b\in B$, write $\pi_b\colon S^B\to S$ for the $b$-th projection. By definition of measure, we have
\begin{align}\label{eq:measure-closed-set}
\M(X,S)=\pi_0^{-1}(0) \ \cap \bigcap_{a\wedge b=0}\{f\in S^B\mid f(a\vee b)=f(a)+ f(b)\}.
\end{align}
The set $\pi_0^{-1}(0)$ is closed because so is $\{0\}\subseteq S$. Further, for each $a,b\in B$,
\[\{f\in S^B\mid f(a\vee b)=f(a)+ f(b)\}\] 
is closed since it is the equaliser of the continuous maps
\[\begin{tikzcd}
S^B \arrow[yshift=3pt]{rr}{\pi_{a\vee b}} \arrow[yshift=-3pt]{rr}[swap]{\pi_{a}+ \, \pi_{b}} & & S
\end{tikzcd}\]
into the Hausdorff space $S$. 
Here, $\pi_{a}+\pi_{b}$ is the composition of the continuous product map $p_i^{-1}(f)\colon S^B\to S^2$ with the continuous map $+\colon S^2\to S$. Thus, equation \eqref{eq:measure-closed-set} exhibits $\M(X,S)$ as an intersection of closed subsets of $S^B$.
\end{proof}
The previous lemma shows that the spaces of measures sit in the category $\BStone$. In fact, they can be seen as limit objects in this category, in the following sense. Let $X$ be a Boolean space, and $\mathcal{D}=\{X_i\mid i\in D\}$ the codirected diagram of the finite discrete spaces `under' $X$ (i.e.\ $\mathcal{D}$ is obtained from the comma category $X \downarrow (\Setfin\to\BStone)$ by forgetting the component which specifies the map from $X$). Each function $\phi_{ij}\colon X_i\to X_j$ in this diagram yields a continuous map $\S{\phi_{ij}}\colon S^{X_j}\to S^{X_i}$, where $S^{X_i}$ and $S^{X_i}$ are equipped with the product topologies. Since a product of Boolean spaces is again Boolean, we get a codirected diagram $\mathcal{D}'=\{S^{X_i}\mid i\in D\}$ in $\BStone$.
\begin{lemma}
The limit in $\BStone$ of the diagram $\mathcal{D}'=\{S^{X_i}\mid i\in D\}$ is the space of measures $\M(X,S)$.
\end{lemma}
\begin{proof}
For each $i\in D$, the corresponding continuous function $\pi_i\colon X\to X_i$ yields a map $\pi_i^*\colon \M(X,S)\to\M(X_i,S)$ sending a measure to its pushforward along $\pi_i$. That is, for $\mu\in \M(X,S)$ and $b$ a clopen of $X_i$, $\pi_i^*\mu(b)=\mu(\pi_i^{-1}(b))$. It is not difficult to see that $\pi_i^*$ is continuous. Composing with the continuous map 
\[
\M(X_i,S)\to S^{X_i}, \ \mu\mapsto \big( x\mapsto \mu(x) \big)
\]
(the expression $\mu(x)=\mu(\{x\})$ makes sense because $X_i$ is discrete), we get a continuous map $\M(X,S)\to S^{X_i}$ that we denote again by $\pi_i^*$. For every function $\phi_{ij}\colon X_i\to X_j$ in $\mathcal{D}$ we have a commutative diagram
\[\begin{tikzcd}
\M(X_i,S) \arrow{r}{\phi_{ij}^*} \arrow{d} & \M(X_j,S) \arrow{d} \\
S^{X_i} \arrow{r}{\S{\phi_{ij}}} & S^{X_j}
\end{tikzcd}\] 
so that $\S{\phi_{ij}}\circ \pi_i^*=\pi_j^*$. We claim that $\{\pi_i^*\colon \M(X,S)\to S^{X_i}\mid i\in D\}$ is a limit cone in $\BStone$.

Let $Y$ be a Boolean space, and $\{f_i\colon Y\to S^{X_i}\mid i\in D\}$ a cone of continuous functions. We show that there exists a unique continuous function $\xi\colon Y\to\M(X,S)$ such that $f_i=\pi_i^*\circ \xi$ for every $i\in D$. Given a clopen $b$ of $X$, consider its characteristic function $\chi_b\colon X\to \2$ into the discrete two-element space $\2=\{0,1\}$. Then there exists $i_b\in D$ such that $\pi_{i_b}=\chi_b$. Define
\[
\xi\colon Y\to\M(X,S), \ y\mapsto \big(b\mapsto f_{i_b}(y)(1), \ \text{where} \ f_{i_b}(y)\in S^{\2}\big).
\]
It is not difficult to see that $\xi$ is well-defined. Further, it is continuous because for every $b\in B$ and every clopen $U\subseteq S$, $\xi^{-1}(\langle b,U\rangle)=(\pi_1\circ f_{i_b})^{-1}(U)$, where $\pi_1\colon S^{\2}\to S$ is the evaluation at $1$. To see that $f_i=\pi_i^*\circ \xi$ for every $i\in D$, fix such an $i$ and consider an arbitrary $x\in X_i$. It suffices to show that $\pi_x\circ f_i=\pi_x\circ \pi_i^*\circ \xi$, where $\pi_x\colon S^{X_i}\to S$ is the $x$-th projection. If $b=\pi_i^{-1}(x)$, we have a commutative diagram as follows.
\[\begin{tikzcd}
{} & X \arrow{dl}[swap]{\pi_i} \arrow{dr}{\pi_{i_b}} & \\
X_i \arrow{rr}{\chi_{\{x\}}} & {} & \2
\end{tikzcd}\]
Therefore, $\S{\chi_{\{x\}}}\circ \pi_i^*=\pi_{i_b}^*$. It follows that
\[
\pi_x\circ f_i=\pi_1\circ \S{\chi_{\{x\}}}\circ f_i=\pi_1\circ f_{i_b}=\pi_1\circ \pi_{i_b}^*\circ \xi=\pi_1\circ \S{\chi_{\{x\}}}\circ \pi_i^*\circ \xi=\pi_x\circ \pi_i^*\circ \xi,
\]
as was to be shown. For the uniqueness of $\xi$, suppose $\zeta\colon Y\to\M(X,S)$ is another continuous map such that $f_i=\pi_i^*\circ\zeta$ for every $i\in D$. Then, for all $y\in Y$ and $b\in B$,
\[
\zeta(y)(b)=\pi_1\circ\pi_{i_b}^*\circ \zeta(y)=\pi_1\circ f_{i_b}(y)=\pi_1\circ \pi_{i_b}^*\circ \xi(y)=\xi(y)(b)
\]
showing that $\zeta=\xi$.
\end{proof}
We will see in Lemma \ref{l:measures-strongly-continuous} below that the $S$-semimodule structure on $\M(X,S)$ is compatible with the Boolean topology in a strong sense. Recall from Definition \ref{d:semiring-and-semimodule} that a semimodule over $S$ is given by an Abelian monoid $M$, along with a `scalar multiplication'
\begin{align*}
\alpha\colon S\times M\to M
\end{align*}
satisfying certain compatibility conditions. Suppose $M$ is equipped with a topology making the monoid operation continuous, that is $M$ is a topological monoid. If $\alpha$ is separately continuous, i.e.\ the functions $\alpha(s,-)\colon M\to M$ are continuous for each $s\in S$, then $M$ is called a \emph{topological $S$-semimodule}. Further,
\begin{definition}\label{d:strongly-continuous}
An $S$-semimodule $M$ is \emph{strongly continuous} if the scalar multiplication $\alpha$ of $S$ on $M$ is not only separately continuous, but also jointly continuous. That is, $\alpha\colon S\times M\to M$ is continuous with respect to the product topology on $S\times M$.
\end{definition}
Not every topological $S$-semimodule is strongly continuous, as the next example shows.
\begin{example}\label{ex:prof-semimod-not-strong}
We give an example of a finite and discrete $S$-semimodule which is not strongly continuous.
Denote by $A$ the semilattice on the set $\{0,1,\omega\}$ whose order is $0<1<\omega$.
Recall from Example \ref{ex:one-point-extended} the profinite semiring $(\Nt,+,\cdot,0,1)$. The obvious action of $\N$ on $A$ (obtained by regarding $A$ as an Abelian monoid) can be extended to an action $\alpha\colon\Nt\times A\to A$ of $\Nt$ on $A$ by setting $\alpha(\infty,0)=0$, and $\alpha(\infty,1)=\alpha(\infty,\omega)=\omega$. This action yields a structure of $\Nt$-semimodule on $A$. If $A$ is equipped with the discrete topology, then the scalar multiplication is obviously separately continuous. However, it is not jointly continuous. Indeed, one has
\[
\alpha^{-1}(\omega)=(\infty,1)\cup (\Nt\setminus\{0\})\times \{\omega\},
\]
which is not clopen because $\infty$ is not an isolated point of $\Nt$.
\end{example}
The spaces of measures $\M(X,S)$ turn out to be strongly continuous, hence topological, $S$-semimodules.
\begin{lemma}\label{l:measures-strongly-continuous}
For any Boolean space $X$, $\M(X,S)$ is a strongly continuous $S$-semimodule. 
\end{lemma}
\begin{proof}
Let $X$ be an arbitrary Boolean space with dual algebra $B$. To prove that $\M(X,S)$ is a topological monoid it suffices to show that, for each $b\in B$, the composition
\[\begin{tikzcd}
\M(X,S)\times \M(X,S) \arrow{r}{+} & \M(X,S) \arrow{r}{\ev_b} & S
\end{tikzcd}\] 
is continuous, where $\ev_b$ is the evaluation map defined in \eqref{eq:evaluation-map}.
In turn, this follows from the commutativity of the next diagram, and the fact that $+\colon S\times S\to S$ is continuous.
\[\begin{tikzcd}[column sep=2cm, row sep=0.8cm]
\M(X,S)\times \M(X,S) \arrow{r}{\ev_b\times \ev_b} \arrow{d}[swap]{+}  & S\times S \arrow{d}{+} \\
\M(X,S) \arrow{r}{\ev_b} & S
\end{tikzcd}\]
The same argument, mutatis mutandis, shows that the function $S\times \M(X,S)\to\M(X,S)$ taking $(s,\mu)$ to $s\cdot\mu$ is continuous. Therefore, $\M(X,S)$ is a strongly continuous $S$-semimodule.
\end{proof}
By Lemmas \ref{l:measures-are-boolean} and \ref{l:measures-strongly-continuous}, for any Boolean space $X$, $\M(X,S)$ is a strongly continuous topological $S$-semimodule on a Boolean space. In \cite[Proposition 7.5]{SZ2017} it is shown that every such semimodule is the inverse limit of a system of finite and discrete strongly continuous $S$-semimodules. In particular, $\M(X,S)$ is a profinite $S$-semimodule. However, $\M(X,S)$ is \emph{not}, in general, the free profinite $S$-semimodule on $X$ (cf.\ Remark \ref{r:quotient-of-strongly-cont}). Nonetheless, we will see in Theorem \ref{t:free-on-idempotent-profinite-rig} that $\M(X,S)$ enjoys a universal property relative to those Boolean topological $S$-semimodules that are strongly continuous. In order to prove the latter theorem we need a preliminary result, Lemma \ref{l:functions-sit-densely} below, relating finitely supported functions and measures on $X$.

Let $X$ be a Boolean space. Recall from equation \eqref{eq:finite-supp-functions} the set $\S|X|$ of finitely supported $S$-valued functions on $X$. Every $f\in\S|X|$ gives a measure on $X$, namely
\begin{align}\label{eq:integral-of-a-function}
\int{f}\colon B\to S, \ \ b\mapsto \int_b{f}=\sum_{x\in b}{f(x)}.
\end{align}
(Throughout, we identify a sum over the empty set with the identity element $0$). The `integration map' $f\mapsto \int{f}$ allows us to identify $\S|X|$ with a dense subset of $\M(X,S)$. 
\begin{lemma}\label{l:functions-sit-densely}
The map $\S|X|\to\M(X,S)$ sending $f$ to $\int{f}$, defined as in \eqref{eq:integral-of-a-function}, is injective with dense image.
\end{lemma}
\begin{proof}
To prove the injectivity, assume $f,g$ are distinct elements of $\S|X|$ and pick $x\in X$ such that $f(x)\neq g(x)$. Write $\sigma$ for the union of the supports of $f$ and $g$, and note that $x\in \sigma$. Since $X$ is Boolean, there is a clopen $b\in B$ such that $b\cap \sigma=\{x\}$. Therefore
\[
\int_b{f}=f(x)\neq g(x) =\int_b{g},
\]
showing that the assignment $f\mapsto \int f$ is injective.
With respect to the density, we must prove that every non-empty basic open subset of $\M(X,S)$ contains a measure of the form $\int f$, for some $f\in \S|X|$. In view of equation \eqref{eq:basic-opens-measures}, such a basic open can be written as
\[
O=\langle b_1,U_1\rangle\cap\cdots\cap\langle b_m,U_m\rangle
\]
where $b_1,\ldots, b_m\in B$, and $U_1,\ldots,U_m$ are clopens of $S$. Let $\{c_1,\ldots,c_n\}$ be the clopen partition of the  set $\bigcup_{i=1}^{m}{b_i}$ induced by the covering $\{b_1,\ldots,b_m\}$, and assume without loss of generality that each $c_j$ is non-empty. In other words, the $c_j$'s are the atoms of the Boolean subalgebra of $B$ generated by the $b_i$'s. Fix an element $x_j\in c_j$ for each $j=1,\ldots,n$. Since $O$ is not empty, it contains a measure $\mu$. Define a function $f\colon X\to S$ with support $\{x_1,\ldots,x_n\}$ such that $f(x_j)=\mu(c_j)$ for each $j$. By finite additivity of $\mu$ we have $\int_{b_i}{f}=\mu(b_i)$ for all $i=1,\ldots,m$, so that $\int f\in O$.
\end{proof}

We are now ready to prove the main result of this section, which provides a characterisation of the profinite algebra $\M(X,S)$ by means of a universal property. Let us say that a strongly continuous $S$-semimodule is \emph{profinite} if it is the inverse limit of finite and discrete strongly continuous $S$-semimodules. As observed after Lemma \ref{l:measures-strongly-continuous}, $\M(X,S)$ is a profinite strongly continuous $S$-semimodule. The next theorem shows that $\M(X,S)$ is free on $X$ with respect to this structure.
\begin{theorem}\label{t:free-on-idempotent-profinite-rig}
Let $S$ be a profinite semiring. For any Boolean space $X$, the collection $\M(X,S)$ of all the $S$-valued measures on $X$ is the free profinite strongly continuous $S$-semimodule on $X$.
\end{theorem}
\begin{proof}
Let $\eta_X\colon X\to \M(X,S)$ be the continuous function sending $x$ to the measure $\mu_x$ concentrated in $x$, i.e.\ $\mu_x(b)=1$ if $x\in b$, and $\mu_x(b)=0$ otherwise. We will prove that $\M(X,S)$ satisfies the universal property with respect to the map $\eta_X$. That is, for every profinite strongly continuous $S$-semimodule $Y$ and continuous function $f\colon X\to Y$, there exists a unique continuous homomorphism of $S$-semimodules $g\colon \M(X,S)\to Y$ such that the following triangle commutes.
\begin{equation}\label{eq:extension-to-free-algebra}
\begin{tikzcd}
X \arrow{r}{\eta_X} \arrow{dr}[swap]{f} & \M(X,S) \arrow[dashed]{d}{g} \\
 & Y
\end{tikzcd}
\end{equation}
By Lemma \ref{l:functions-sit-densely} the function $\S|X|\to\M(X,S)$, mapping $f$ to $\int{f}$, is injective and has dense image. 
Observe that any measure on $X$ of the form $\int{f}$, for $f\in\S|X|$, is a finite linear combination with coefficients in $S$ of measures concentrated in a point.
Thus any two continuous homomorphisms making diagram \eqref{eq:extension-to-free-algebra} commute must coincide on the image of $\S|X|\to\M(X,S)$. 
Since the latter is dense in $\M(X,S)$, and $Y$ is Hausdorff, there is at most one solution to the diagram above.

To exhibit such a solution, we proceed as follows. Let $\{\pi_i\colon Y\to Y_i\mid i\in I\}$ be a cone of continuous homomorphisms defining $Y$ as the inverse limit of the finite and discrete strongly continuous $S$-semimodules $Y_i$, and set \[f_i=\pi_i\circ f\colon X\to Y_i.\] We will define a cone of continuous homomorphisms $\{g_i\colon \M(X,S)\to Y_i\mid i\in I\}$ such that the induced limit map $\M(X,S)\to Y$ provides the desired solution. For each $i\in I$ consider the square
\begin{equation*}
\begin{tikzcd}
X \arrow{r}{\eta_X} \arrow{d}[swap]{f} & \M(X,S) \arrow{d}{f_i^*} \\
Y_i & \M(Y_i,S) \arrow{l}[swap]{h_{Y_i}}
\end{tikzcd}
\end{equation*}
where $f_i^*\colon \M(X,S)\to\M(Y_i,S)$ sends a measure $\mu$ to its pushforward with respect to $f_i$, i.e.\ $f_i^*\mu(b)=\mu(f_i^{-1}(b))$ for every clopen $b$ of $Y_i$, and
\[
\forall \nu\in\M(Y_i,S), \ \ h_{Y_i}(\nu)= \sum_{y\in Y_i}{\nu(y)\cdot y}.
\]
Here, $\nu(y)$ stands for $\nu(\{y\})$, and the expression makes sense because $Y_i$ is discrete. It is not difficult to see that the pushforward maps $f_i^*$ are continuous homomorphisms of $S$-semimodules. Suppose for a moment that the $h_{Y_i}$'s are also continuous homomorphisms. Then, for each $i\in I$, $g_i=h_{Y_i}\circ f_i^*$ would be a continuous homomorphism satisfying $g_i\circ\eta_X=f_i$. Indeed, 
\[
\forall x\in X \ \ (g_i\circ\eta_X)(x)= \sum_{y\in Y_i}{f_i^*\mu_x(y)\cdot y}=f_i(x)
\]
because $f_i^*\mu_x$ is the measure on $Y_i$ concentrated in $f_i(x)$. If $g\colon\M(X,S)\to Y$ is the continuous homomorphism of $S$-semimodules induced by the cone $\{g_i\colon \M(X,S)\to Y_i\mid i\in I\}$, we have $g\circ \eta_X=f$. That is, $g$ is a solution to diagram \eqref{eq:extension-to-free-algebra}.
Hence, it remains to show that each $h_{Y_i}$ is a continuous homomorphism. To improve readability, we write $Z$ instead of $Y_i$, and assume that $Z=\{y_1,\ldots,y_n\}$. We only check that $h_Z$ is continuous, for the preservation of the algebraic structure is easily verified. 
Consider the composition
\[
\gamma\colon (S\times Z)^n\to Z^n\to Z
\]
where the first map sends $((\ell_1,z_1),\ldots,(\ell_n,z_n))$ to $(\ell_1\cdot z_1,\ldots,\ell_n\cdot z_n)$, and the second one sends $(z_1,\ldots,z_n)$ to $z_1+\cdots+z_n$. Since $Z$ is a strongly continuous $S$-semimodule, $\gamma$ is a continuous function. For any $z\in Z$, let $T_z$ be the closed subset of $S^n$ obtained by projecting the clopen set $\gamma^{-1}(z)\subseteq (S\times Z)^n$ onto the $S$-coordinates. Then, one has
\[
h_Z^{-1}(z)=\big\{\nu\in\M(Z,S)\mid \sum_{y\in Z}{\nu(y)\cdot y}=z\big\}=(\ev_{y_1}\times\cdots\times\ev_{y_n})^{-1}(T_z)
\]
for any $z\in Z$, where $\ev_{y_1}\times\cdots\times\ev_{y_n}\colon \M(Z,S)\to S^n$. The latter function is continuous, whence $h_Z^{-1}(z)$ is a closed subset of $\M(Z,S)$, showing that the function $h_Z$ is continuous. This concludes the proof.
\end{proof}
We conclude the section by showing that, in general, $\M(X,S)$ is not the free profinite $S$-semimodule on $X$. This is due to the fact that separate continuity of the scalar multiplication on an $S$-semimodule does not imply joint continuity. However, it clearly does if $S$ if finite, for then the two notions coincide. The latter case will be treated in the next section.
\begin{remark}\label{r:quotient-of-strongly-cont}
Let $X$ be any Boolean space. We claim that every profinite $S$-semimodule which is a continuous homomorphic image of $\M(X,S)$, is a strongly continuous $S$-semimodule. Note that this implies that $\M(X,S)$ cannot be the free profinite $S$-semimodule on $X$, for otherwise every profinite $S$-semimodule would be strongly continuous (and we know by Example \ref{ex:prof-semimod-not-strong} that this is not the case).

To settle the claim, let $A$ be a profinite $S$-semimodule and $f\colon \M(X,S)\twoheadrightarrow A$ a continuous surjective homomorphism. Write $\alpha\colon S\times \M(X,S)\to \M(X,S)$, $\beta \colon S\times A\to A$ for the scalar multiplications on $\M(X,S)$ and $A$, respectively. We have the following commutative square.
\[\begin{tikzcd}
S\times \M(X,S) \arrow{r}{\alpha} \arrow[twoheadrightarrow]{d}[swap]{\textrm{id}_S\times f} & \M(X,S) \arrow[twoheadrightarrow]{d}{f} \\
S\times A \arrow{r}{\beta} & A 
\end{tikzcd}\]
We must prove that $\beta$ is continuous. Note that $f$, and hence also $\textrm{id}_S\times f$, are topological quotients. Thus, for every open subset $U\subseteq A$, $\beta^{-1}(U)$ is open if, and only if, $\beta\circ(\textrm{id}_S\times f)^{-1}(U)$ is open in $S\times \M(X,S)$. In turn, the latter set is open because the diagram commutes and $f\circ \alpha$ is continuous.
\end{remark}

\section{The case of finite semirings: the main result}\label{s:finite-case}
We aim to give a concrete representation of the free profinite $S$-semimodule on a Boolean space $X$. In view of Proposition \ref{p:hat-profinite-algebra} the latter is isomorphic to $\p{\S}{X}$, where $\p{\S}$ is the profinite monad of the $\Set$-monad $\S$ associated with the semiring $S$. In Theorem \ref{t:free-on-idempotent-profinite-rig} we saw that, if $S$ is profinite, then the algebra $\M(X,S)$ of all the $S$-valued measures on $X$ is the free profinite \emph{strongly continuous} $S$-semimodule on $X$. 
Provided $S$ is finite, every topological $S$-semimodule is strongly continuous. Therefore, we obtain the following theorem as a corollary.
\begin{theorem}\label{t:main-S-finite}
Let $S$ be a finite semiring, and $X$ a Boolean space. Then $\p{\S}{X}$, the free profinite $S$-semimodule on $X$, is isomorphic to the algebra $\M(X,S)$ of all the $S$-valued measures on $X$.\qed
\end{theorem}

In the remainder of the section we indicate how one could give a direct proof of Theorem \ref{t:main-S-finite}, exploiting the finiteness of the semiring. 
Throughout the section we assume $S=(S,+,\cdot,0,1)$ is a finite semiring. We first describe the dual algebra of $\p{\S}{X}$ in terms of the dual algebra of the Boolean space $X$. Recall from \eqref{eq:finite-supp-functions} the set $\S|X|$ of finitely supported $S$-valued functions on $X$, and the integration map $\S|X|\to \M(X,S)$, $f\mapsto \int{f}$ of \eqref{eq:integral-of-a-function}.
\begin{lemma}\label{l:characterisation-B-hat}
Let $X$ be a Boolean space with dual algebra $B$. The algebra $\p{B}$ dual to $\p{\S}{X}$ is isomorphic to the Boolean subalgebra of $\P(\S|X|)$ generated by the elements of the form
\[
[b,k]=\big\{f\in\S|X|\mid \int_b{f}=k\big\},
\]
for $b\in B$ and $k\in S$.
\end{lemma}
\begin{proof}
Let $X$ be a Boolean space. Then $X$ is the limit of the codirected diagram $\{X_i\mid i\in I\}$ of its finite continuous images. Write $\pi_i\colon X\to X_i$ for the $i$-th limit map. Since $S$ is finite, by equation \eqref{eq:limit-formula-locally-finite} the Boolean space $\p{\S}X$ is homeomorphic to the inverse limit of the finite discrete spaces $\S{X_i}$. Let $p_i\colon \p{\S}X\to \S{X_i}$ be the $i$-th limit map. As observed in Section \ref{s:codensity-profinite-monads}, under these hypotheses the `comparison map' $\t_X\colon \S|X|\to|\p{\S}{X}|$ from Definition \ref{d:nt-tau} is injective and satisfies
\begin{equation}\label{eq:restriction}
\S{\pi_i}=p_i\circ\t_X
\end{equation}
 for each $i\in I$. To improve notation, we identify $\S|X|$ with its image under $\t_X$.
The dual algebra $\p{B}$ of $\p{\S}X$ is generated by the set 
\[
\{p_i^{-1}(U)\mid i\in I, \ U\in\P(\S{X_i})\}.
\]
Note that the set above actually coincides with $\p{B}$, because it is already closed under the Boolean operations. 
Indeed, it is clearly closed under taking complements, and it is closed under finite intersections because the diagram $\{X_i\mid i\in I\}$ is codirected.
Now, consider the Boolean algebra homomorphism $\gamma\colon\p{B}\to \P(\S|X|)$ sending a clopen of $\p{\S}X$ to its restriction to the subset $\S|X|$. In view of equation \eqref{eq:restriction}, this map can be equivalently described as
\[
\gamma\colon \p{B}\to \P(\S|X|), \ p_i^{-1}(U)\mapsto (\S{\pi_i})^{-1}(U)
\]
and it is injective precisely because $\t_X$ has dense image by Lemma \ref{l:dense-image}. Therefore, $\p{B}$ is isomorphic to the Boolean subalgebra $\gamma(\p{B})$ of $\P(\S|X|)$. In turn, for every $(\S{\pi_i})^{-1}(U)\in \gamma(\p{B})$ we have
\begin{align*}
(\S{\pi_i})^{-1}(U)&=\bigcup_{f\in U}{\big\{g\in \S|X|\mid \S{\pi_i}(g)=f \big\}} \\
&=\bigcup_{f\in U}{\bigcap_{x\in X_i}{\big\{g\in \S|X|\mid \int_{\pi_i^{-1}(x)}{g=f(x)}\big\}}} \\
&=\bigcup_{f\in U}{\bigcap_{x\in X_i}{[\pi_i^{-1}(x),f(x)]}}.
\end{align*}
To conclude that $\p{B}$ is isomorphic to the subalgebra of $\P(\S|X|)$ generated by the elements of the form $[b,k]$, where $b$ ranges over the clopens of $X$ and $k\in S$, it suffices to show that each $[b,k]$ belongs to $\gamma(\p{B})$. Assume without loss of generality that $b\neq \emptyset, X$ and consider its characteristic map $\chi_b\colon X\to\2$. Then there is $i_b\in I$ such that $\pi_{i_b}=\chi_b$. It follows that $[b,k]=(\S{\pi_{i_b}})^{-1}(U)\in \gamma(\p{B})$, where $U=\{f\colon \2\to S\mid f(1)=k\}$.
\end{proof}
Now, let $\phi$ be a point of $\p{\S}{X}$, i.e.\ an ultrafilter on the Boolean algebra $\p{B}$. By the previous lemma, for each $b\in B$, $\{[b,k]\mid k\in S\}$ is a finite set of pairwise disjoint elements of $\p{B}$ whose join is the top element. Thus we can define a function
\begin{equation}\label{eq:morphisms-measures}
\p{\S}{X}\to \M(X,S), \ \phi\mapsto \mu_{\phi}
\end{equation}
where, for each $b\in B$, we define $\mu_{\phi}(b)$ to be the unique $k\in S$ satisfying $[b,k]\in \phi$. It is not difficult to see that each $\mu_{\phi}$ is, indeed, a measure. This correspondence is injective because the elements of the form $[b,k]$ generate the Boolean algebra $\p{B}$ (see Lemma \ref{l:characterisation-B-hat}). 

On the other hand, let $\mu\colon B\to S$ be a measure on $X$. We will exhibit an ultrafilter $\phi$ on $\p{B}$ such that $\mu=\mu_{\phi}$. Consider the set $F=\{[b,\mu(b)]\mid b\in B\}\subseteq \P(\S|X|)$. Observe that $[b,k]=\emptyset$ if, and only if, $b=0$ and $k\neq 0$. Hence, the empty set does not belong to $F$ because $\mu(0)=0$. Moreover, for every $b_1,\ldots,b_n\in B$,
\[
[b_1,\mu(b_1)]\cap \cdots \cap [b_n,\mu(b_n)]\neq \emptyset
\]
by additivity of $\mu$, i.e.\ $F$ is a filter base. Let $\phi$ be the (proper) filter generated by $F$. It is enough to prove that $\phi$ is an ultrafilter, for then $\mu=\mu_{\phi}$. Since the $[b,k]$'s generate $\p{B}$, it suffices to show that $[b,k]\notin \phi$ implies $[b,k]^c\in \phi$. Assume $[b,k]\notin \phi$. Then $k\neq\mu(b)$ entails $[b,\mu(b)]\subseteq [b,k]^c$, whence $[b,k]^c\in \phi$.

This shows that the map in \eqref{eq:morphisms-measures} is a bijection. One can check that it is also a continuous homomorphism of $S$-semimodules, so that we recover the result in Theorem \ref{t:main-S-finite}.
\begin{theorem}\label{t:measures-representation-concrete}
Let $S$ be a finite semiring, and $X$ a Boolean space. The map in \eqref{eq:morphisms-measures} yields a continuous isomorphism of $S$-semimodules between $\p{\S}{X}$ and $\M(X,S)$. Therefore, the free profinite $S$-semimodule over $X$ is isomorphic to the algebra $\M(X,S)$ of all the $S$-valued measures on $X$.\qed
\end{theorem}
Upon identifying an element of $\p{\S}{X}$ with the corresponding measure on $X$, the `comparison map' $\t_X\colon \S|X|\to |\p{\S}{X}|$ of Definition \ref{d:nt-tau} can be concretely described as the integration function
\[
\t_X\colon\S|X|\to \M(X,S), \ f\mapsto \int{f}.
\]
The latter map is an embedding with dense image and, for each $b\in B$ and $k\in S$, the closure of the subset 
\[
[b,k]=\big\{f\in\S|X|\mid \int_b{f}=k\big\}
\] 
of $\M(X,S)$ is the subbasic clopen subset \[\langle b,k\rangle=\{\mu\in\M(X,S)\mid \mu(b)=k\}.\]
Moreover, for any continuous map $h\colon X\to Y$ and measure $\mu\in\M(X,S)$, the continuous homomorphism $\p{\S}{h}\colon \M(X,S)\to\M(Y,S)$ sends a measure $\mu$ on $X$ to its pushforward with respect to $h$. That is, \[\p{\S}{h}(\mu)\colon b\mapsto \mu(h^{-1}(b))\] for every clopen $b$ of $Y$.
Further, recall from \eqref{eq:Stone-Cech-functor} the adjunction $|-|\colon\BStone\leftrightarrows\Set\cocolon \beta$. Since adjoints compose, the free profinite $S$-semimodule on a set $A$ is isomorphic to $\M(\beta A,S)$, where $\beta A$ is the Stone-{\v C}ech compactification of the discrete space $A$. Note that an element of $\M(\beta A,S)$ is a finitely additive function $\P(A)\to S$, i.e.\ the measurable subsets of $\beta A$ are in bijection with the subsets of $A$.
\begin{remark}
Theorem \ref{t:main-S-finite} yields, in the case of the two-element distributive lattice $\2$, a representation of the Vietoris space $\V{X}$ of a Boolean space $X$ as the space of $\2$-valued measures over $X$. This may be compared with the representations by Shapiro \cite{Shapiro1992} and Radul \cite{Radul1997} of $\V{X}$, for $X$ a compact Hausdorff space, in terms of real-valued functionals.
\end{remark}
\section{The case of profinite idempotent semirings: algebras of continuous functions}\label{s:idempotent-case}
In this final section we show that, if $S$ is a profinite idempotent semiring, then all the $S$-valued measures are uniquely given by continuous density functions (Theorem \ref{t:continuous-function-representation-profinite-idempotent}). By Theorem \ref{t:main-S-finite}, this yields a representation of the free profinite $S$-semimodule on a Boolean space $X$ in terms of continuous $S$-valued functions on $X$, provided $S$ is a finite idempotent semiring.

Suppose $(S,+,\cdot,0,1)$ is a profinite semiring which is \emph{idempotent}, i.e.\ it satisfies $s+s=s$ for every $s\in S$.
Any idempotent semiring is equipped with a natural partial order $\leq$ defined by $s\leq t$ if, and only if, there is $u$ such that $s+u=t$. The operation $+$ is then a join-semilattice operation with identity $0$. Accordingly, we write $\vee$ instead of $+$. 
In particular, a profinite idempotent semiring is a topological join-semilattice on a Boolean space.\footnote{Although we shall not need this fact, we remark that the topological semilattices whose underlying spaces are Boolean, are precisely the profinite semilattices \cite{Numakura1957}.}
Next we recall some basic facts about such topological semilattices that we will use in the following. We warn the reader that, while we work with join-semilattices, most of the literature (cf.\ \cite{Compendium2003,HMS74,Johnstone1986}) deals with meet-semilattices. 
\begin{definition}
An element $k$ in a complete lattice $L$ is \emph{compact} if, for every subset $S\subseteq L$ such that $k\leq \bigvee S$, there is a a finite subset $F\subseteq S$ with $k\leq \bigvee F$. An \emph{algebraic lattice} is a complete lattice in which every element is the supremum of the compact elements below it.
\end{definition}
Let $L$ be a \emph{directed complete poset} (\emph{dcpo}, for short). That is, $L$ is a poset in which every directed subset admits a supremum. A subset $U\subseteq L$ is called \emph{Scott open} if it is upward closed and, for every directed subset $D\subseteq L$,
\[
\bigvee D\in U \ \Rightarrow \ D\cap U\neq\emptyset.
\] 
The collection of all Scott open subsets is a topology, the \emph{Scott topology} of $L$. Further, the \emph{lower topology} on $L$ is the topology generated by the sets of the form $(\up{x})^c$ for $x\in L$.
\begin{definition}
The \emph{Lawson topology} of a dcpo $L$ is the smallest topology containing both the Scott topology and the lower topology.
\end{definition}
The following theorem identifies the topology of a topological meet-semilattice on a Boolean space as the Lawson topology. For a proof see, e.g., \cite[Theorem VI-3.13]{Compendium2003}.
\begin{theorem}
Let $L$ be a topological meet-semilattice with $1$ whose underlying space is Boolean. Then $L$ is an algebraic lattice and its topology is the Lawson topology.\qed
\end{theorem}
In the case of the profinite idempotent semiring $S$, the previous theorem entails that $S$ is a complete lattice in which every element is the infimum of the \emph{co-compact} elements above it (the concept of co-compact element is the order-dual of that of compact element). Thus the topology of $S$, the \emph{dual Lawson topology} (i.e., the Lawson topology of the order-dual of $S$), has as basic opens the sets of the form 
\begin{align}\label{eq:base-prof-semil}
\down{k}\cap (\down{l_1})^c\cap\cdots\cap(\down{l_n})^c,
\end{align}
where $k,l_1,\ldots,l_n$ are co-compact elements of $S$. 
Every set of the form $\down{k}$, with $k$ co-compact, is clopen \cite[Theorem II.3.3]{HMS74}; this shows that the sets in \eqref{eq:base-prof-semil} provide a basis of clopens for $S$. Finally, any directed subset of $S$ considered as a net converges to a unique limit, namely its least upper bound. Similarly, for codirected subsets and greatest lower bounds (see, e.g., \cite[{}II.1]{HMS74}). 

In view of the completeness of $S$, for each measure $\mu\in\M(X,S)$ we can define a function 
\begin{equation*}
\delta_{\mu}\colon X\to S, \ x\mapsto \bigwedge_{x\in b}{\mu(b)}
\end{equation*}
that intuitively provides the value of the measure $\mu$ at a point. 
In general, the functions $\delta_{\mu}\colon X\to S$ are not continuous with respect to the dual Lawson topology of $S$. However, they are continuous with respect to the \emph{dual Scott topology}, i.e.\ the Scott topology of the order-dual of $S$. The latter coincides with the topology of all those open sets (in the dual Lawson topology) which are downward closed, cf.\ \cite[Proposition III-1.6]{Compendium2003}. 
\begin{definition}\label{d:open-down-sets-topology}
Let $S$ be a profinite idempotent semiring. We define $\d{S}$ to be the topological space obtained by equipping the underlying set of $S$ with the dual Scott topology.
\end{definition}
\begin{lemma}\label{l:densities-are-continuous}
Let $S$ be a profinite idempotent semiring. For every Boolean space $X$ and measure $\mu\in\M(X,S)$, $\delta_{\mu}\colon X\to \d{S}$ is a continuous function. 
\end{lemma}
\begin{proof}
Let $\mu$ be a measure on $X$, and $U$ an open down-set of $S$. We must prove that the preimage
\[
\delta_{\mu}^{-1}(U)=\big\{x\in X\mid \bigwedge_{x\in b}{\mu(b)}\in U\big\}
\]
is open.
Note that the set $\{\mu(b)\mid x\in b\}$ is codirected. If its infimum belongs to $U$, which is dual Scott open, there must exist $b\in B$ containing $x$ and satisfying $\mu(b)\in U$. Thus
\[
\delta_{\mu}^{-1}(U)\subseteq\bigcup{\{b\in B\mid \mu\in \langle b,U\rangle\}}.
\]
The converse inclusion holds because $U$ is a down-set. This shows that $\delta_{\mu}^{-1}(U)$ is open in $X$. 
\end{proof}
Let $\C(X,\d{S})$ denote the set of all the $S$-valued functions on $X$ which are continuous with respect to the dual Scott topology of $S$. This can be regarded as a semilattice, with respect to the pointwise order. Similarly for $\M(X,S)$. In view of the previous lemma, there is a function 
\begin{align}\label{eq:density-g1}
\delta\colon\M(X,S)\to\C(X,\d{S}), \ \mu\mapsto \delta_{\mu}
\end{align}
which is readily seen to be monotone. 
In the converse direction, since $S$ is complete, for every function $f\colon X\to S$ and clopen $b$ of $X$ we can define the integral of $f$ over $b$ as 
\begin{align*}
\int_b{f}=\bigvee_{x\in b}{f(x)}.
\end{align*}
This notion of integration with values in an idempotent semiring is well-known, and it is studied in particular in idempotent analysis (see, e.g., \cite{KM1997}). So we have the integration map 
\begin{align}\label{eq:int-g2}
\int\colon \C(X,\d{S})\to\M(X,S), \ \ f\mapsto \big(b\mapsto \int_{b}{f}\big)
\end{align}
which is also monotone.
\begin{proposition}\label{p:Galois-connection}
Let $S$ be a profinite idempotent semiring. The maps 
\begin{equation*}
\delta\colon \M(X,S)\rightleftarrows \C(X,\d{S})\cocolon \int
\end{equation*}
defined in \eqref{eq:density-g1} and \eqref{eq:int-g2} form an adjoint pair, where $\delta$ is upper adjoint and $\int$ is lower adjoint. 
\end{proposition}
\begin{proof}
We must prove that, for any $\mu\in\M(X,S)$ and $f\in\C(X,\d{S})$, we have $\int f\leq \mu \Leftrightarrow f\leq \delta_{\mu}$. In turn, this  follows at once from the definitions of $\int{f}$ and $\delta_{\mu}$.
\end{proof}
The set $\C(X,\d{S})$ of continuous $\d{S}$-valued functions on $X$ carries a natural structure of $S$-semimodule, where both the monoid operation and the scalar multiplication are defined pointwise. With respect to this structure, the functions $\delta\colon \M(X,S)\rightleftarrows \C(X,\d{S})\cocolon \int$ are seen to be homomorphisms of $S$-semimodules. Moreover, they are continuous if the set $\C(X,\d{S})$ is equipped with the topology generated by the sets of the form
\[
\big\{f\in\C(X,S^{\downarrow})\mid \int_{b}{f}\in U\big\},
\]
for $b$ a clopen of $X$ and $U$ an open subset of $S$.
We will see that, in fact, the adjoint pair in Proposition \ref{p:Galois-connection} provides an isomorphism of topological algebras between $\M(X,S)$ and $\C(X,\d{S})$.
We first show that $\delta_{\mu}$ can be regarded as the \emph{density function} of the measure $\mu$. 
\begin{lemma}\label{l:density-functions}
Let $S$ be a profinite idempotent semiring, and $X$ a Boolean space with dual algebra $B$. For every $\mu\in\M(X,S)$ and $b\in B$, $\mu(b)=\int_{b}{\delta_{\mu}}$.
\end{lemma}
\begin{proof}
Fix $b\in B$. We show that $\mu(b)$ is the limit in $S$ of the directed set
\begin{align*}
N=\big\{\bigvee_{x\in F}{\delta_{\mu}(x)}\mid F\in \Pfin(b)\big\},
\end{align*}
considered as a net. Since $\int_{b}{\delta_{\mu}}$ is also a limit for this net, it will follow that $\mu(b)=\int_{b}{\delta_{\mu}}$ because $S$ is Hausdorff. Let $k,l_1,\ldots,l_n$ be co-compact elements of $S$ such that the basic open set
\[
U= \down{k}\cap (\down{l_1})^c\cap\cdots\cap(\down{l_n})^c.
\]
contains $\mu(b)$. We prove that the net $N$ is eventually in the open neighbourhood $U$ of $\mu(b)$. Note that, for each $x\in b$, $\delta_{\mu}(x)$ is below $\mu(b)$ whence it belongs to $\down{k}$. So it suffices to find, for every $i\in\{1,\ldots,n\}$, a point $x_i\in b$ such that $\delta_{\mu}(x_i)\in (\down{l_i})^c$, for then every element of $N$ above $\bigvee_{i=1}^n{\delta_{\mu}(x_i)}$ will belong to $U$. Assume by contradiction that there exists $i\in\{1,\ldots,n\}$ with
\[
b\cap \delta_{\mu}^{-1}((\down{l_i})^c)=\emptyset.
\]
That is, $b\subseteq \delta_{\mu}^{-1}(\down{l_i})$. Since $\down{l_i}$ is clopen, for each $x\in b$ there is an open neighbourhood $U_x$ of $\delta_{\mu}(x)$ contained in $\down{l_i}$. By definition, $\delta_{\mu}(x)$ is the limit of the net $\{\mu(b)\mid x\in b\}$, so for every $x\in b$ there is $b_x\in B$ such that $x\in b_x$ and $\mu(b_x)\in U_x$. We can assume without loss of generality that each $b_x$ is contained in $b$. Then the clopen covering $\{b_x\mid x\in b\}$ of $b$ has a finite subcover $\{b_{x_1},\ldots, b_{x_p}\}$. For every $j\in\{1,\ldots,p\}$ we have $\mu(b_{x_j})\in U_x\subseteq \down{l_i}$, thus
\[
\mu(b)=\mu(b_{x_1})\vee\cdots\vee\mu(b_{x_p})\leq l_i,
\]
a contradiction.
\end{proof}
\begin{theorem}\label{t:continuous-function-representation-profinite-idempotent}
Let $S$ be a profinite idempotent semiring, and $X$ a Boolean space. Then the continuous homomorphisms of $S$-semimodules \[\delta\colon \M(X,S)\rightleftarrows \C(X,\d{S})\cocolon \int\] of \eqref{eq:density-g1}--\eqref{eq:int-g2} are inverse to each other. Thus $\M(X,S)$, the algebra of all the $S$-valued measures on $X$, is isomorphic to the algebra $\C(X,\d{S})$ of all the continuous $\d{S}$-valued functions on $X$.
\end{theorem}
\begin{proof}
In view of Lemma \ref{l:density-functions} we know that $\int\circ \, \delta$ is the identity of $\M(X,S)$, for every Boolean space $X$. 
It remains to prove that, whenever $f\colon X\to\d{S}$ is a continuous function, the measure $\mu=\int{f}$ satisfies $f=\delta_{\mu}$. That is, for each $x\in X$,
\begin{equation*}
f(x)=\bigwedge{\big\{\int_b{f}\mid x\in b, \ b\in B\big\}},
\end{equation*}
where $B$ is the dual algebra of $X$. Regarding the codirected set 
\[
N=\big\{\int_b{f}\mid x\in b, \ b\in B\big\}
\] 
as a net, this is equivalent to saying that the limit of $N$ is $f(x)$. Consider co-compact elements $k,l_1,\ldots,l_n$ of $S$ such that the basic open set
\[
U= \down{k}\cap (\down{l_1})^c\cap\cdots\cap(\down{l_n})^c
\]
contains $f(x)$. We must prove that $N$ is eventually in $U$. Of course we have $\int_b{f}\in (\down{l_1})^c\cap\cdots\cap(\down{l_n})^c$ for every $b$ containing $x$. So it suffices to find a clopen $b'\in B$ such that $x\in b'$ and $\int_{b'}{f}\leq k$, for then every element of $N$ below ${\int_{b'}{f}}$ will belong to $U$. Since the function $f$ is continuous with respect to the dual Scott topology of $S$, and $\down{k}$ is dual Scott open, $f^{-1}(\down{k})$ is an open neighbourhood of $x$. Let $b'\in B$ be a clopen satisfying $x\in b'\subseteq f^{-1}(\down{k})$. Then
\[
\int_{b'}{f}\leq k,
\]
as was to be proved.
\end{proof}
Note that, if the semiring $S$ is finite, the dual Scott topology on $S$ is simply the down-set topology, i.e.\ the Alexandroff topology of the order-dual of $S$. In this situation, the previous theorem has the following immediate corollary.
\begin{theorem}\label{t:continuous-function-representation-1}
Let $S$ be a finite idempotent semiring, and $X$ a Boolean space. Then $\p{\S}{X}$, the free profinite $S$-semimodule on $X$, is isomorphic to the algebra $\C(X,\d{S})$ of all the continuous $\d{S}$-valued functions on $X$.
\end{theorem}
\begin{proof}
This follows from Theorems \ref{t:main-S-finite} and \ref{t:continuous-function-representation-profinite-idempotent}.
\end{proof}
\begin{remark} 
If $S$ is the two-element distributive lattice $\2$, then $\d{S}$ is homeomorphic to the 
Sierpi{\' n}ski space. We thus recover from Theorem \ref{t:continuous-function-representation-1} the classical representation of the Vietoris space $\V X$ of a Boolean space $X$ as the semilattice of all the continuous functions from $X$ into the Sierpi{\' n}ski space. 
\end{remark}

\section*{Acknowledgements}
I would like to thank my Ph.D.\ advisor Mai Gehrke, and Daniela Petri\c{s}an, for many helpful discussions on the topic of this paper.
The research reported here has been made possible by financial support from Sorbonne Paris Cit\'e (PhD agreement USPC IDEX -- REGGI15RDXMTSPC1GEHRKE), and from the European Research Council (ERC) under the European Union's Horizon 2020 research and innovation programme (grant agreement No.670624). Finally, I am grateful to the anonymous referee for their useful suggestions, which improved the paper.

\end{document}